\documentclass[a4paper,12pt]{article}

\usepackage{graphicx}
\usepackage{amsmath, amsfonts, amsthm, amssymb, enumerate, hyperref,
fancyhdr, fullpage, color}
\usepackage[all]{xy}

\newcommand{\ignore}[1]{}

\newcommand{\floor}[1]{\left \lfloor #1 \right \rfloor}

\newtheorem{thm}{Theorem}
\newtheorem{lemma}{Lemma}

\newtheorem{conj}{Conjecture}

\theoremstyle{definition}

\title{Multi-Switch: a Tool for Finding Potential Edge-Disjoint $1$-factors}

\author{Tyler Seacrest \\ The University of Montana Western \\ 710 S Atlantic St \\ Dillon, MT 59725, USA \\ \texttt{tyler.seacrest@umwestern.edu}}

\date{\today}

\begin{document}

\maketitle

\abstract{Let $n$ be even,  let $\pi = (d_1, \ldots, d_n)$ be a graphic degree sequence, and let $\pi - k = (d_1 - k, \ldots, d_n - k)$ also be graphic.  Kundu proved that $\pi$ has a realization $G$ containing a $k$-factor, or $k$-regular graph.  Another way to state the conclusion of Kundu's theorem is that $\pi$ \emph{potentially} contains a $k$-factor.

Busch, Ferrara, Hartke, Jacobsen, Kaul, and West conjectured that more was true:  $\pi$ potentially contains $k$ edge-disjoint $1$-factors.  Along these lines, they proved $\pi$ would potentially contain edge-disjoint copies of a $(k-2)$-factor and two $1$-factors.  

We follow the methods of Busch et al.\ but introduce a new tool which we call a multi-switch.  Using this new idea, we prove that $\pi$ potentially has edge-disjoint copies of a $(k-4)$-factor and four $1$-factors.  We also prove that $\pi$ potentially has ($\floor{k/2} + 2$) edge-disjoint $1$-factors, but in this case cannot prove the existence of a large regular graph.}

\section{Introduction}
Listing out all the degrees of a graph forms its \emph{degree sequence} or \emph{degree list}.  Conversely, given a list $\pi = (d_1, d_2, \ldots, d_n)$, we say that $\pi$ is \emph{graphic} if it is the degree sequence of a graph $G$, and that $G$ \emph{realizes} $\pi$.  We can also add an integer to $\pi$, where $\pi + k = (d_1 + k, d_2 + k, \ldots, d_n + k)$.   A \emph{$k$-factor} is a $k$-regular subgraph of a graph.  Thus, a $1$-factor is a perfect matching, and a $2$-factor is a spanning collection of disjoint cycles.

Given a sequence $\pi = (d_1, d_2, \ldots, d_n)$, it is said $\pi$ \emph{potentially} has property $P$ if there is at least one realization of $\pi$ that has property $P$.  One beautiful result along these lines, due to Kundu~\cite{Kundu73}, characterizes those degree sequences that potententially have a $k$-factor.
\begin{thm}[Kundu~\cite{Kundu73}]
A graphic sequence $\pi$ potentially has a $k$-factor if and only if $\pi-k$ is graphic.
\end{thm}
\noindent See Chen~\cite{Chen88} for a short and elegant proof of Kundu's theorem.

 We say that $\pi$ has \emph{even length} if the length of the list (and hence number of vertices intended in the resulting graph) is even.  Having an even number of vertices is an obvious necessary condition for a graph to contain a $1$-factor.  Busch, Ferrara, Hartke, Jacobson, Kaul, and West~\cite{BuschEtAl12} conjectured that if $\pi$ is of even length, then the $k$-factor from Kundu's theorem can be decomposed into edge-disjoint $1$-factors.   

\begin{conj}[Busch et al.~\cite{BuschEtAl12}]
\label{conj:buschetal}
The sequences $\pi$ and $\pi-k$ are graphic and of even length if and only if $\pi$ potentially has $k$ edge-disjoint $1$-factors.
\end{conj}

Using the vocabulary of edge-colorings, we could also state this result as saying the $k$-factor from Kundu's theorem is potentially of class $1$, or is potentially $1$-factorizable.  

Along these lines, they strengthened Kundu's theorem by proving:

\begin{thm}[Busch et al.~\cite{BuschEtAl12}]
\label{thm:busch}
If $\pi$ and $\pi-k$ are graphic and of even length, then $\pi$ has a realization containing edge-disjoint copies of one $(k-2)$-factor and two $1$-factors.
\end{thm}

It is important to note a classical theorem by Petersen~\cite{Petersen1891}, which says that for even $k$, any $k$-regular graph can be decomposed into two-factors.

\begin{thm}[Petersen's $2$-factor Theorem~\cite{Petersen1891}]
\label{thm:petersen}
If $k$ is even, then a $k$-regular graph contains $k/2$ edge-disjoint $2$-factors.
\end{thm}

Hence, the $k$-factor from Kundu's theorem decomposes into $2$-factors without even considering alternate realizations of $\pi$.  However, finding $1$-factors is a bit more challenging.   Indeed -- finding even two $1$-factors in a graph (i.e.\ without degree sequence considerations) is NP-complete~\cite{LevenGalil83}.

In this note, we improve the Busch et al.\ result a bit farther, showing

\begin{thm}
\label{thm:4-1-factors}
If $\pi$ and $\pi-k$ are graphic and of even length, then $\pi$ has a realization containing edge-disjoint copies of one $(k-4)$-factor and four $1$-factors.
\end{thm}

We can also get many more $1$-factors with the same hypothesis, at the expense of not leaving the rest  of the $k$-factor intact.

\begin{thm}
\label{thm:k/2-1-factors}
If $\pi$ and $\pi-k$ are graphic and of even length, then $\pi$ has a realization containing $(\floor{k/2} + 2)$ edge-disjoint $1$-factors.
\end{thm}

The new idea is that of a multi-switch.  Using the visualization of different colors representing the different $1$-factors, a multi-switch is finding several paths of length two between a pair of vertices and swapping the two colors on the edges within each path.  While switching moves with multiple colors have been used before, such as proving Vizing's Theorem (see the proof of Theorem 7.1.10 in~\cite{West96}), they can be difficult to work with.  The multi-switch, explained in more detail in Section~\ref{sec:multi}, is relatively simple and seems to be an important tool when dealing with packing edge-disjoint graphs in degree sequence problems.   With this tool in hand, we prove Theorems~\ref{thm:4-1-factors} and \ref{thm:k/2-1-factors} in Sections~\ref{sec:4-1-factors} and \ref{sec:k/2-1-factors} respectively.

\section{Multi-switch}
\label{sec:multi}

Consider a degree list $\pi$.  Given a $k$ such that $\pi$ and $\pi-k$ are graphic, suppose $\pi$ has a realization $G$ containing edge-disjoint subgraphs $A_1, A_2, \ldots, A_\ell$, where $A_i$ is a regular subgraph of degree $m_i$, and $\sum m_i = k$.   We will call $G$ a \emph{realization with regular subgraphs}.   Let $B = G - \cup A_i$, and let $W = K_n - G$.  Hence the $A_i$, $B$ and $W$ partition $K_n$.  For ease of visualization, assume edges in $B$ are colored black, edges in $W$ are colored white, and edges in $A_i$ are colored with color $c_i$.  We will call $B$, $W$, and the $A_i$ the \emph{colored subgraphs}.

The multi-switch is a generalization of a two-switch.   One situation yielding a two-switch is as follows: suppose we have a realization of a degree sequence $\pi$ with two vertices $u$ and $v$ where $\deg_G(u) \geq \deg_G(v)$. Suppose further there exists a path of length two via a third vertex $w$ such that $uw$ is a non-edge and $wv$ is an edge.   Suppose our goal is to swap edges $uw$ and $wv$, so $uw$ becomes an edge and $wv$ becomes a non-edge.  Since the degree of $u$ is at least that of $v$, there must be another vertex $z$ such that $uz$ is an edge and $zv$ is a non-edge.  Hence, if we swap $uw$ and $wv$, and simultaneously swap $uz$ and $zv$, we have made the desired change while maintaining the same degree at each vertex.  

\begin{figure}
\begin{center}
\includegraphics[scale=1]{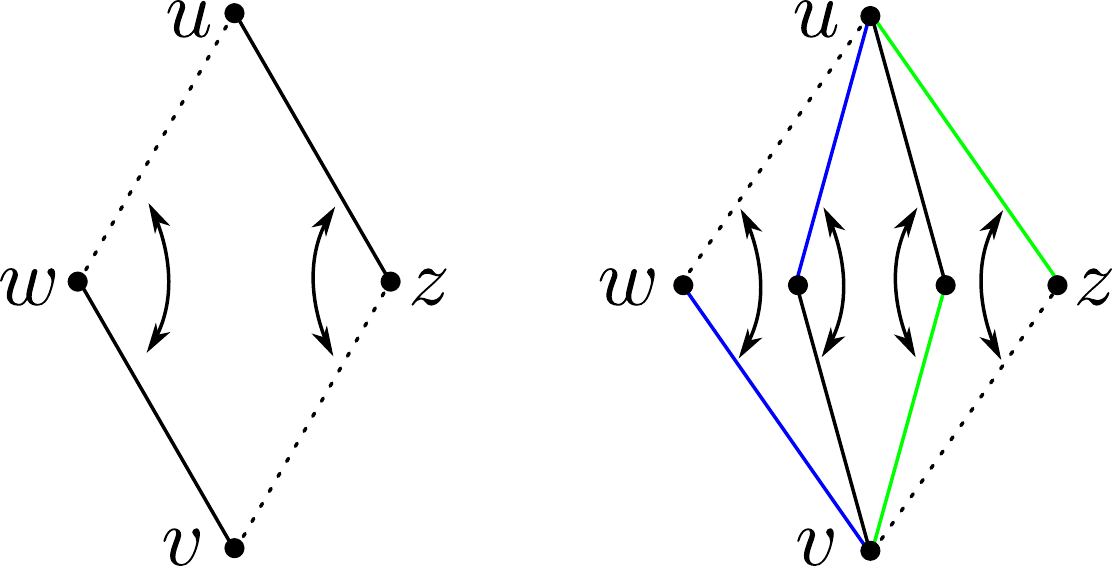}
\end{center}
\caption{\label{fig:two-multi}  A diagram showing a two-switch on the left, and a multi-switch on the right.  The dotted lines represent white edges.}
\end{figure}

The multi-switch is basically this same idea, generalized as to possibly involve more paths of length two between $u$ and $v$. See Figure~\ref{fig:two-multi}.

\begin{lemma}[Multi-switch]
\label{lemma:multi-1}
Consider a degree sequence $\pi$ and a realization with regular subgraphs $G$.   Given vertices $u, v$ with $\deg_G(u) \geq \deg_G(v)$, let $x_1$ and  $y_1$ be edges in a path of length two from $u$ to $v$.  Suppose $x_1$ is white, and $y_1$ is some other color $c$.  Let $z_1$ be a $c$-colored edge incident to $u$, and if possible let $z_2$ and $z_3$ be additional $c$-colored edges incident to $u$ and $v$ respectively.  see Figure~\ref{fig:lemma}.  

Then there is a switching move that 
\begin{itemize}
	\item swaps the colors of $x_1$ and $y_1$,
	\item swaps the colors of exactly one of $\{z_1, z_2\}$ with another edge incident to $v$,
	\item involves only  one white edge  ($x_1$) incident to $u$ and only one white edge incident to $v$,
	\item does not involve $z_3$ or any edge not incident to $u$ or $v$, and
	\item maintains the degree of every vertex in every color. 
\end{itemize}
\end{lemma}

\begin{figure}
\begin{center}
\includegraphics[scale=1]{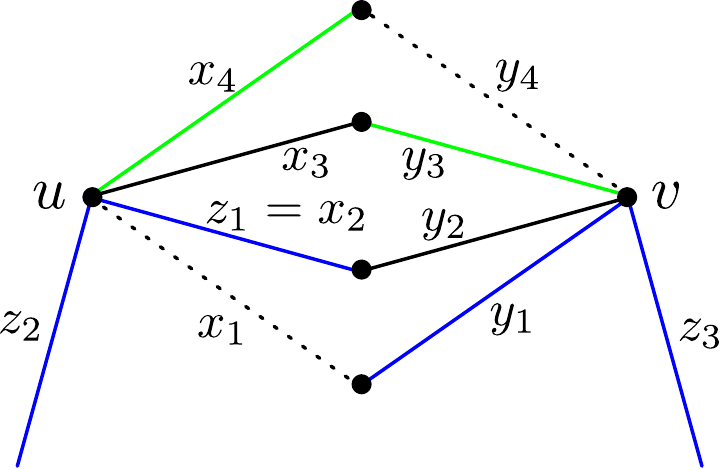}
\end{center}
\caption{\label{fig:lemma}  A diagram showing the various edge names from Lemma~\ref{lemma:multi-1}.}
\end{figure}

\begin{proof}
Set $x_2 = z_1$, and let $y_2$ be the edge (possibly white) such that the edges $x_2 y_2$ form a path from $u$ to $v$ in $K_n$.  Suppose $y_2$ is a non-white color $c'$.  Then, since $u$ has at least as many incident edges colored $c'$ as $v$, there must be some edge $x_3$ incident to $u$ of color $c'$.  Let $y_3$ be the edge such that $x_3 y_3$ is a path of length two from $u$ to $v$.  Similarly, as long as $y_3$ is non-white, there must be an $x_4$ the same color as $y_3$ incident to $u$, and we can therefore find a path $x_4 y_4$ from $u$ to $v$.  By repeating this argument, we can achieve a list $(x_1, y_1, x_2, y_2, \ldots, x_r, y_r)$, where $x_{i+1}$ has the same color as $y_i$ for all $i$.   This list can be extended until we reach a $y_r$ that is white.  At which point, by switching the colors of $x_i$ and $y_i $ for all $i$, we will maintain the degree of each color and $x_1$ and $y_1$ will be flipped, as desired.  

Note that $y_r$ is the first and only white edge on the list incident to $v$, and hence $x_1$ is the only white edge incident to $u$ on the list.

Note futher that if $z_3$, the vertex of color $c$ incident to $v$, is ever used, we can then use $z_2$ as the next $x_i$, and continue.  For example, say $z_3 = y_j$;  then $x_{j+1} = z_2$.  Then, instead of swapping $x_i$ and $y_i$ for all $i$, we can swap $x_i$ and $y_i$ for $i = 1$ and $i \geq j+1$.   If $z_3$ is never used or does not exist, then we will avoid using $z_2$.   In this way, the requirements regarding exactly one of $\{z_1, z_2\}$ being used and $z_3$ never being used are satisfied.
\end{proof}

Note that we can also achieve the same result using a black edge instead of a white edges if the degree inequality is reversed.  The proof is symmetric and therefore omitted.

\begin{lemma}
\label{lemma:multi-2}
Lemma~\ref{lemma:multi-1} is also true in the case where $\deg_G(u) \leq \deg_G(v)$ and $x_1$ is black.
\end{lemma}

Also note that the existence of $z_2$ and $z_3$ in the statement of Lemma~\ref{lemma:multi-1} is to ensure that not too many edges swap colors in certain circumstances.  If no such $z_2$ or $z_3$ exist, the conclusion of the lemmas still hold with reference to $z_2$ and $z_3$ omitted.

\section{Matchings in regular graphs}

To prove Theorem~\ref{thm:4-1-factors}, we will also need a structural result very similar to  Lemma 3.7 of Busch et al.\ \cite{BuschEtAl12},  which was used in the proof of Theorem~\ref{thm:busch}.  They used the Edmonds-Gallai Structural Theorem.  In the interest of being more self-contained, we give an edge-switching proof instead.  Given a matching $M$ in a graph $G$, an odd cycle $C$ is \emph{fully matched} if every vertex in $C$ is matched with another vertex in $C$, except for one.

\begin{lemma}
\label{lemma:odd}
Let $G$ be a $k$-regular graph. Then there exists a maximum size matching $M$ such that every vertex not covered by $M$ is contained on a fully matched odd cycle, and any two vertices not covered by $M$ are contained in disjoint fully matched odd cycles.
\end{lemma}

\begin{proof}
Let $M$ be a maximum matching of $G$, and let $v$ be a vertex uncovered by the matching.  Given any even length path $P$ that starts with $v$ and alternates edges not in $M$ with edges in $M$, we can toggle all the edges in $P$ from in $M$ to not in $M$, and vice versa, to move the uncovered vertex from $v$ to the last vertex of $P$. 

Let $D$ be the set of all possible locations for the uncovered vertex by such switching moves starting at $v$.    Let $N$ be all the neighbors of vertices in $D$.   Every vertex in $N$ must be matched with a vertex in $D$, or otherwise we could increase the size of $D$.    Notice that $|D| > |N|$, since every vertex in $N$ is matched with a vertex in $D$, and $D$ also has $v$.  Suppose $D$ is an independent set.  All of $D$'s neighbors are then in $N$, and therefore $D$ has $k|D|$ edges leaving it, but there is no way for $N$ to absorb all of these edges.  Hence, $D$ contains a non-matched edge $e$ connecting vertices $x$ and $y$.  

Suppose $P_1$ is the alternating path that ends at $x$, and $P_2$ is the alternating path that ends at $y$, and let $z$ be the last vertex $P_1$ and $P_2$ have in common.  This creates an odd cycle $C$ going from $x$ to $z$ along $P_1$, then from $z$ to $y$ along $P_2$, and finally going from $y$ back to $x$ using $e$.  $C$ is then a fully matched cycle with $z$ the only vertex not matched with another vertex on $C$.  By switching edges along the path $P_1$ from $v$ to $z$, we can change the uncovered vertex from $v$ to a vertex $z$, and then the uncovered vertex will be on a fully matched odd cycle.  See Figure~\ref{fig:path-switch}.  Note if any such path $P_1$ intersects with a fully matched odd cycle of another unmatched vertex $u$, then we could extend $P$ to create an alternating path between two unmatched vertices, contradicting the maximality of $M$.  Hence none of the other fully match odd cycles belonging to other unmatched vertices are disturbed.

\begin{figure}
\begin{center}
\includegraphics[scale=1]{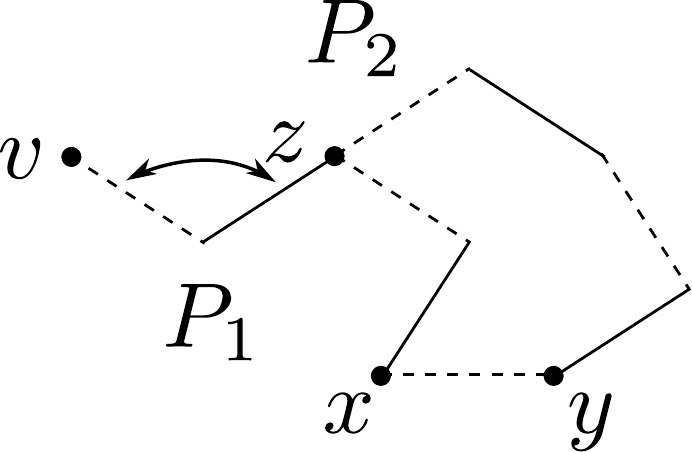}
\end{center}
\caption{\label{fig:path-switch} This shows an example of the paths $P_1$ and $P_2$ from the proof of Lemma~\ref{lemma:odd}, where $P_1$ is the entire path from $v$ to $x$, and $P_2$ is the entire path from $v$ to $y$.  Here, the solid edges are those in the matching $M$ while the dotted edges are those not in the matching.  By swapping the edges along the path from $v$ to $z$, we can move the unmatched vertex to a fully matched odd cycle.}
\end{figure}

Finally, note that for any two vertices uncovered by $M$, their fully matched odd cycles must be disjoint.  Otherwise, there would be a alternating path from one to the other, and we could increase the size of the matching.
\end{proof}

\section{Proof of Theorem~\ref{thm:4-1-factors}}
\label{sec:4-1-factors}

Consider a realization of $\pi$ containing a $(k-i)$-factor and $i$ edge-disjoint $1$-factors.  Suppose we choose $\pi$ to first  maximize $i$, and then to maximize the number of edges in a matching of the $(k-i)$-factor.   Suppose, by way of contradiction, that $i \leq 3$.  

For ease of discussion, supoose the $(k-i)$-factor has blue edges, edges not colored blue or contained in a $1$-factor in $\pi$ are black, and the rest of the edges of $K_n$ not included in the realization of $\pi$ are white.  

By Lemma~\ref{lemma:odd}, there exists a maximum size matching $M$ in the blue graph such that two vertices missed by $M$ are contained on fully matched odd cycles $C_1$ and $C_2$.  There cannot be any blue edges between $C_1$ and $C_2$, since otherwise we could extend the matching $M$ to include these missed vertices.

Since $C_1$ is an odd cycle, we must have three consecutive vertices $u_1, u_2, u_3$ along $C_1$ such that $\deg(u_1) \geq \deg(u_2) \geq \deg(u_3)$.  Similarly, we have three consecutive vertices $v_1, v_2, v_3$ along $C_2$ such that $\deg(v_1) \geq \deg(v_2) \geq \deg(v_3)$.  Without loss of generality, assume $\deg(u_2) \geq \deg(v_2)$.  Then we have $\deg(u_1) \geq \deg(u_2) \geq \deg(v_2) \geq \deg(v_3)$.  

Let $e_1 = u_1 v_2$, $e_2 = u_1 v_3$, $e_3 = u_2 v_2$ and $e_4 = u_2 v_3$.   If any of $\{e_1, e_2, e_3, e_4\}$ are black or white edges, then we can perform multi-switches that will allow us to extend $M$.  For example, suppose $e_3$ is a white edge.  Then apply Lemma~\ref{lemma:multi-1} using $u = u_2$, $v = v_3$, $x_1 = e_3$, $y_1 = v_2 v_3$, $z_1 = u_1 u_2$, and $z_2$ and $z_3$ being the other incident blue edge of $u$ along $C_1$ and $v$ along $C_2$ respectively.  After applying the lemma, the multi-switch will swap either $z_1$ and $y_1$, or it will swap $z_2$ and $y_1$, and no other edges of $C_1$ or $C_2$ were affected. Thus, we can now extend $M$ to cover $C_1$ and $C_2$ completely by using $y_1$ and then alternating edges around both cycles.

Hence, none of these four edges are black or white, so they are all contained in $1$-factors.  Therefore, there are two edges in the same $1$-factor.  These edges must be parallel:  either $e_1$ and $e_4$ or $e_2$ and $e_3$.  Either way, a simple two-switch will combine the cycles $C_1$ and $C_2$ into a large even cycle, and we can increase the size of the matching by alternating edges around this combined cycle.   By contradiction, we have $i \geq 4$.  \qed

\section{Proof of Theorem~\ref{thm:k/2-1-factors}}
\label{sec:k/2-1-factors}

By Theorem~\ref{thm:4-1-factors}, there exists a realization of $\pi$ containing a $(k-4)$-factor and four edge-disjoint $1$-factors.    If $k$ is even, we will partition the $(k-4)$-factor into $\frac{k}{2} - 2$ two-factors using Petersen's two-factor theorem.  If $k$ is odd, we will take three $1$-factors and partition the remaining $(k-3)$-factor into $\frac{k-1}{2} - 1$ two-factors using Petersen's theorem.  This will be our initial setup.  During the course of the proof, we will at  any given stage have a realization of $\pi$ containing several edge-disjoint $1$-factors and $2$-factors.  Our goal will be to repeatedly change one of the $2$-factors into a $1$-factor.

Let us focus on a particular $2$-factor, $F$.   Let $M$ be a matching that covers as many vertices as possible of the graph, using edges from $F$ and black edges that go between distinct odd cycles of $F$.  If $M$ covers all the vertices of the graph, then we remove $F$ from our list of $2$-factors, and add $M$ to our list of $1$ factors, and move on to the next $2$-factor.

Suppose $M$ does not cover all the vertices of the graph.  Any even length cycles of $F$ can be easily covered by $M$ using edges in $F$, so this leaves the odd cycles.  Given a pair of odd cycles, if there is a black edge between the odd cycles, then if we use this black edge in the matching, we can alternate around the two odd cycles to cover both completely.    Therefore, suppose we have a pair of odd cycles $C_1$ and $C_2$ with no black edges between them.

We will think of all the edges of $C_1$ and $C_2$ as being black for the time being.   Without loss of generality, assume there exists a vertex $u \in C_1$ and $v \in C_2$ with $\deg(u) \leq \deg(v)$.  Let $w$ be a neighbor of $u$ along $C_1$.   Apply Lemma~\ref{lemma:multi-2} with $u = u$, $v = v$, $x_1 = uw$, $y_1 = wv$, $z_1$ as a neighbor of $u$ the same color as $y_1$, and we need not set $z_2$ or $z_3$ specifically.   After this switch, $uw$ will no longer be a black edge, but $wv$ will be a black edge.  We can now increase $M$ so that it covers $C_1$ and $C_2$ using $wv$ which is now black, and then alternate using black edges around $C_1$ and $C_2$.  Note that none of the edges in $C_1$ and $C_2$ needed to extend the matching were affected, since at most two black edges are used in Lemma~\ref{lemma:multi-2}, and these are incident to $w$ and $v$.   Further note that the use of Lemma~\ref{lemma:multi-2} does not affect any other cycles besides $C_1$ and $C_2$ in $F$, since all the affected edges are incident to $u$ and $w$ and hence are not contained in other cycles in $F$.

Repeating with all pairs of  odd cycles in $F$, this gives an additional $1$-factor in the graph for our list.   Repeating this argument with all  of the $2$-factors gives a total of $\frac{k}{2} - 2$ one-factors if $k$ is even, or $\frac{k-1}{2} - 1$ one-factors if $k$ is odd.   Adding in the four or three $1$-factors we removed in the beginning gives a total of $\floor{\frac{k}{2}} + 2$ one-factors. \qed

\bibliography{potential}
\bibliographystyle{plain}

\end{document}